\documentclass[english,11pt]{amsart}
\usepackage{amsopn,amsthm,amsfonts,amsmath,amssymb,amscd}
\usepackage{babel}
\usepackage{nicematrix}
\usepackage{mathtools}

\newtheorem{Theorem}{Theorem}[section]
\newtheorem{Lemma}[Theorem]{Lemma}

\newtheorem{Corollary}[Theorem]{Corollary}
\newtheorem{Example}[Theorem]{Example}

\newcommand{\Tr}{\mathrm{tr}}

\newcommand{\OO}{\mathcal{O}}

\begin{document}

\title{Products of Nilpotent and Idempotent Matrices over Finite Local Rings}

\author{David Dol\v{z}an}
\address{Department of Mathematics, Faculty of Mathematics and Physics, University of Ljubljana, Jadranska 19, SI-1000 Ljubljana, Slovenia; and Institute of Mathematics, Physics and Mechanics, Jadranska 19, SI-1000 Ljubljana, Slovenia}
\email{david.dolzan@fmf.uni-lj.si}

\thanks{The author acknowledges the financial support from the Slovenian Research Agency  (research core funding No. P1-0222)}

\begin{abstract}
Let $R$ be a finite commutative local principal ring. We study
products of nilpotent and idempotent matrices in $M_2(R)$. We show
that every product of nilpotent and idempotent matrices is either a
product of idempotents or a product of nilpotents. We then consider
IN- and NI-matrices, that is, matrices which can be written as a
product of an idempotent and a nilpotent matrix in the respective
orders. We prove that the classes of IN- and NI-matrices in $M_2(R)$
coincide and give an explicit description of their common class.
Finally, if $|R|=q^n$ and $R/J(R)\cong GF(q)$, we show that
\[
    |\operatorname{IN}(M_2(R))|
    =
    |\operatorname{NI}(M_2(R))|
    =
    q^{3n-2}(q^n+q^2-1).
\]
\end{abstract}

\maketitle 

 \section{Introduction}

Products of special classes of elements in rings and matrix rings have
been studied from various points of view. In particular, products of
idempotent and nilpotent matrices arise naturally in factorization
problems. Products consisting exclusively of idempotents or exclusively of
nilpotents have been considered in the literature. See for example \cite{alahmadi1, alahmadi, caluga, dolzanid, erdos, howie,  jain, laffey, laffey1} for products of idempotents and
\cite{caluga, dolzannil} for products of nilpotents. In our recent study~\cite{dolzanid} of products of idempotents in
$M_2(R)$ over finite commutative local principal rings, the relevant
conjugacy classes were described and enumerated. It is therefore natural
to ask what happens when factors of both types are allowed to occur in
an arbitrary order. Products involving one idempotent and one nilpotent factor have
recently been studied over domains in \cite{caluga2}.

In this paper we study this question for $2$-by-$2$ matrices over a
finite commutative local principal ring $R$. Our first result shows that
allowing both types of factors does not produce any new matrices beyond
those already obtainable using only one of the two types. More
precisely, if $\mathcal I$ and $\mathcal N$ denote the sets of
idempotent and nilpotent matrices in $M_2(R)$, respectively, we prove
that
\[
    \langle\mathcal I\cup\mathcal N\rangle
    =
    \langle\mathcal I\rangle\cup\langle\mathcal N\rangle.
\]
Moreover, every such product containing a nontrivial idempotent factor
is a product of two idempotents.

We next consider more specifically products involving one idempotent
and one nilpotent factor. Following C\u{a}lug\u{a}reanu and Pop~\cite{caluga2},
an element which can be written as $EN$, where $E$ is idempotent and
$N$ is nilpotent, is called an IN-element, while a product $NE$ is
called an NI-element. In general, these two notions need not coincide;
see~\cite{caluga2}. We show that the situation is different
for $2$-by-$2$ matrices over finite commutative local principal rings:
every IN-matrix is an NI-matrix and conversely. More precisely, writing
\[
    M(a,b)=
    \begin{pmatrix}
        a&b\\
        0&0
    \end{pmatrix}
\]
and denoting its conjugacy orbit by $O_{M(a,b)}$, we obtain the
characterization
\[
    \operatorname{IN}(M_2(R))
    =
    \operatorname{NI}(M_2(R))
    =
    N(M_2(R))
    \cup
    \bigcup_{a,b\in R}O_{M(a,b)}.
\]
We also show that every matrix belonging to one of the orbits
$O_{M(a,b)}$ admits both a nontrivial IN-factorization and a
nontrivial NI-factorization.

Finally, this characterization allows us to determine the number of
IN- and NI-matrices. If $|R|=q^n$ and
$R/J(R)\cong GF(q)$, then
\[
    |\operatorname{IN}(M_2(R))|
    =
    |\operatorname{NI}(M_2(R))|
    =
    q^{3n-2}(q^n+q^2-1).
\]
The proof combines the above characterization with the description and
enumeration of the conjugacy orbits $O_{M(a,b)}$ and a count of
nilpotent matrices over finite local rings.

The paper is organized as follows. In Section~2 we recall the necessary
notation and basic facts concerning finite local principal rings and
introduce IN- and NI-matrices. Section~3 contains the main results on
arbitrary products of idempotent and nilpotent matrices, establishes
the equality of the classes of IN- and NI-matrices, and determines
their cardinality.

\bigskip

\bigskip

 \section{Definitions and preliminaries}
\bigskip

All rings in our paper will be finite rings with identity. 
For a ring $R$, $N(R)$ will denote the set of all its nilpotents. The group of units in $R$ will be denoted by $U(R)$ and the Jacobson radical of $R$ by $J(R)$. 

We call a vector $x=(x_1,\ldots,x_n)^T\in R^n$ unimodular if $Rx_1+\cdots+Rx_n=R$.
Let $e_i \in R^n$ denote the $i$-th standard basis vector, whose $i$-th component equals $1$ and all other components equal $0$.

Following \cite{caluga2}, we call an
element $a$ of a ring $R$ an \emph{IN-element} if
$a=EN$ for some idempotent $E\in R$ and some nilpotent $N\in R$.
Symmetrically, $a$ is called an \emph{NI-element} if
$a=NE$
for some nilpotent $N\in R$ and some idempotent $E\in R$.
For matrices, we shall also use the terms \emph{IN-matrix} and
\emph{NI-matrix}, respectively.
We call an IN-factorization $A=EN$ \emph{nontrivial} if $E$ is a
nontrivial idempotent and $N\neq0$. Similarly, an NI-factorization
$A=NE$ is called \emph{nontrivial} if $E$ is a nontrivial idempotent
and $N\neq0$.

We will denote the $2$-by-$2$ matrix ring with entries in a subring $S$ of the ring $R$ by $M_2(S)$, while the group of invertible matrices therein will be denoted by $GL_2(S)$.


For any $a, b \in R$, let $M(a,b)=\left(\begin{array}{cc}
a & b\\
0 & 0
\end{array}\right) \in M_2(R)$. The trace of a matrix $A \in M_2(R)$ will be denoted by $\Tr(A)$.

The group $GL_2(R)$ acts on $M_2(R)$ by conjugation. We will denote the orbit of an element $X \in M_2(R)$ for this action by $\OO_X$.

It follows from \cite[Theorem 2]{ragha} that every finite local ring has order $p^{nr}$ for some prime number $p$ and some integers $n, r$. Furthermore, the Jacobson radical $J(R)$ is of order $p^{(n-1)r}$ and the factor ring $R/J(R)$ is a field with $p^r$ elements (denoted $GF(p^r)$).

We also have the following lemma, which can be found in \cite{dolzanid}.

\begin{Lemma} \cite[Lemma 2.2]{dolzanid}
\label{pid}    
Let $R$ be a finite local principal ring of cardinality $q^{n}$, where $R/J(R) \cong GF(q)$. Then there exists $x \in J(R)$ such that $J(R)^l=(x^l)$ for every $l \in \{0,1,\ldots,n\}$. In particular, $|J(R)^l \setminus J(R)^{l+1}|=q^{n-l-1}(q-1)$ for every $l \in \{0,1,\ldots,n-1\}$.
\end{Lemma}

%
%
%

\bigskip

 \section{Products of nilpotent and idempotent matrices}
\bigskip

\begin{Lemma}
\label{unimod}
Let $R$ be a finite commutative local principal ring. For every
nonzero vector $v=\begin{pmatrix}a\\b\end{pmatrix}\in R^2$
there exist $t\in R$, a unimodular vector $u\in R^2$ 
 and $P\in GL_2(R)$ such that $v=tu$, $Pu=\begin{pmatrix}1\\0\end{pmatrix}$ and
$Pv=\begin{pmatrix}t\\0\end{pmatrix}$.
\end{Lemma}

\begin{proof}
Since $R$ is a finite local principal ring, its ideals are linearly
ordered. Hence either $(a)\subseteq(b)$ or $(b)\subseteq(a)$.

Suppose first that $(b)\subseteq(a)$. Then $b=sa$ for some $s\in R$,
and therefore $v=
    \begin{pmatrix}a\\b\end{pmatrix}
    =
    a\begin{pmatrix}1\\s\end{pmatrix}$.
The vector $u=\begin{pmatrix}1\\s\end{pmatrix}$
is unimodular. Taking
\[
    P=
    \begin{pmatrix}
        1&0\\
        -s&1
    \end{pmatrix}\in GL_2(R),
\]
we obtain
$Pv=
    \begin{pmatrix}
        1&0\\
        -s&1
    \end{pmatrix}
    \begin{pmatrix}a\\sa\end{pmatrix}
    =
    \begin{pmatrix}a\\0\end{pmatrix}$.
Thus the assertion holds with $t=a$.

If $(a)\subseteq(b)$, we proceed similarly. We have $a=sb$ for some $s\in R$, so
$v=b\begin{pmatrix}s\\1\end{pmatrix}$.
Again, $(s,1)^T$ is unimodular, and taking
$P=
    \begin{pmatrix}
        0&1\\
        1&-s
    \end{pmatrix}\in GL_2(R)$
gives
$Pv=
    \begin{pmatrix}b\\0\end{pmatrix}$.
Hence the assertion holds with $t=b$.
\end{proof}

\begin{Lemma}
\label{prodcontid}
Let $R$ be a finite commutative local principal ring and let
$A,B\in M_2(R)$. If $E\in M_2(R)$ is a nontrivial idempotent, then
there exist $a,b\in R$ such that $AEB\in \OO_{M(a,b)}$.
\end{Lemma}

\begin{proof}
Since $R$ is local, every nontrivial idempotent in $M_2(R)$ is
conjugate to $M(1,0)$. Hence there exists $P\in GL_2(R)$ such that
$E=P^{-1}M(1,0)P$.
Therefore
$AEB
    =P^{-1}\bigl((PAP^{-1})M(1,0)(PBP^{-1})\bigr)P$.
It thus suffices to show that
$CM(1,0)D\in \OO_{M(a,b)}$
for arbitrary $C,D\in M_2(R)$ and some $a,b\in R$.

Write $C=
    \begin{pmatrix}
        c_1&c_2\\
        c_3&c_4
    \end{pmatrix}$ and $D=
    \begin{pmatrix}
        d_1&d_2\\
        d_3&d_4
    \end{pmatrix}$.
Then
\[
    CM(1,0)D
    =
    \begin{pmatrix}
        c_1d_1&c_1d_2\\
        c_3d_1&c_3d_2
    \end{pmatrix}
    =
    \begin{pmatrix}c_1\\c_3\end{pmatrix}
    \begin{pmatrix}d_1&d_2\end{pmatrix}.
\]
If $c_1=c_3=0$, the assertion is immediate.
Otherwise, Lemma \ref{unimod} implies that there exist $t\in R$, a unimodular vector
$u\in R^2$ such that $\begin{pmatrix}c_1\\c_3\end{pmatrix}=tu$ and $Q\in GL_2(R)$ such that
$Qu=
    \begin{pmatrix}1\\0\end{pmatrix}$.
It follows that
\[
    Q\bigl(CM(1,0)D\bigr)Q^{-1}
    =
    t
    \begin{pmatrix}1\\0\end{pmatrix}
    \begin{pmatrix}d_1&d_2\end{pmatrix}Q^{-1}
    =
    \begin{pmatrix}
        a&b\\
        0&0
    \end{pmatrix}
    =M(a,b)
\]
for some $a,b\in R$. Hence
$CM(1,0)D\in \OO_{M(a,b)}$, and consequently
$AEB\in \OO_{M(a,b)}$.
\end{proof}

\begin{Theorem}
\label{th1}
Let $R$ be a finite commutative local principal ring. Denote by
$\mathcal I$ and $\mathcal N$ the sets of idempotent and nilpotent
matrices in $M_2(R)$, respectively. Then
\[
    \langle\mathcal I\cup\mathcal N\rangle
    =
    \langle\mathcal I\rangle\cup
    \langle\mathcal N\rangle.
\]
Moreover, if a product of elements of $\mathcal I\cup\mathcal N$
contains a nontrivial idempotent factor, then it is a product of two
idempotents.
\end{Theorem}

\begin{proof}
The inclusion
$\langle\mathcal I\rangle\cup\langle\mathcal N\rangle
    \subseteq
    \langle\mathcal I\cup\mathcal N\rangle$
is obvious.

Conversely, let $A=A_1A_2\cdots A_s$ for some $ A_i\in\mathcal I\cup\mathcal N$.
If all the idempotent factors occurring in this product are equal
to $I$, then they can be omitted. If at least one nilpotent factor
occurs, then $A$ is a product of nilpotents and hence
$A\in\langle\mathcal N\rangle$; otherwise $A=I\in\langle\mathcal I\rangle$.
If one of the idempotent factors is $0$, then $A=0$, and in particular
$A\in\langle\mathcal I\rangle$.

It remains to consider the case where the product contains a
nontrivial idempotent $E$. Write
$A=BEC$
for some $B,C\in M_2(R)$. By Lemma \ref{prodcontid}, there exist
$a,b\in R$ such that $A\in \OO_{M(a,b)}$.
However, every matrix belonging to $\OO_{M(a,b)}$ is a product of two
idempotents by \cite[Lemma 3.3]{dolzanid}. Consequently $A\in\langle\mathcal I\rangle$.
\end{proof}

The following lemma uses the idea of \cite[Lemma 3.3]{dolzanid}.

\begin{Lemma}
\label{mabne}
Let $R$ be a finite commutative local principal ring. Then every
matrix $M(a,b)$ is both an IN-matrix and an NI-matrix. Moreover, it
admits both a nontrivial IN-factorization and a nontrivial
NI-factorization.
\end{Lemma}

\begin{proof}
Denote $J(R)=(x)$. 
We first prove that $M(a,b)$ is a product of a nontrivial idempotent
and a nonzero nilpotent.
If $a\neq0$, let
    $a\in J(R)^l\setminus J(R)^{l+1}$,
and set $l=n$ if $a=0$. Similarly, if $b\neq0$, let 
    $b\in J(R)^k\setminus J(R)^{k+1}$,
and set $k=n$ if $b=0$.

Assume first that $k\leq l$ and that $(a,b)\neq(0,0)$. Choose $u,v\in U(R)$ such that
\[
    a=vx^l,\qquad b=ux^k;
\]
when $a=0$ or $b=0$, the corresponding unit may be chosen
arbitrarily. Define $c=-u^{-1}v^2x^{2l-k}$.
Since $k\leq l$, we have $2l-k\geq0$, and
\[
    bc
    =(ux^k)(-u^{-1}v^2x^{2l-k})
    =-v^2x^{2l}
    =-a^2.
\]
Hence $N_1=
    \begin{pmatrix}
        a&b\\
        c&-a
    \end{pmatrix}$
satisfies
    $N_1^2=
    \begin{pmatrix}
        a^2+bc&0\\
        0&a^2+bc
    \end{pmatrix}
    =0$.
Thus $N_1$ is nilpotent. Since $(a,b)\neq(0,0)$, it is nonzero.
Taking
    $E_1=
    \begin{pmatrix}
        1&0\\
        0&0
    \end{pmatrix}$,
which is a nontrivial idempotent, we obtain
    $E_1N_1
    =
    \begin{pmatrix}
        a&b\\
        0&0
    \end{pmatrix}
    =M(a,b)$.

If $l<k$, then
\[
    a-b=\left(1-uv^{-1}x^{k-l}\right)a\in(a),
\]
so by \cite[Lemma~3.2]{dolzanid}, $\OO_{M(a,b)}=\OO_{M(a,a)}$.
Now
    $M(a,a)
    =
    \begin{pmatrix}
        1&0\\
        0&0
    \end{pmatrix}
    \begin{pmatrix}
        a&a\\
        -a&-a
    \end{pmatrix}$,
and
    $\begin{pmatrix}
        a&a\\
        -a&-a
    \end{pmatrix}^2=0$.
If $a\neq0$, the nilpotent factor is nonzero. Therefore
$M(a,a)$ is a product of a nontrivial idempotent and a nonzero
nilpotent, and the same holds for $M(a,b)$ after conjugation.

Finally, if $a=b=0$, then
    $M(0,0)
    =
    \begin{pmatrix}
        1&0\\
        0&0
    \end{pmatrix}
    \begin{pmatrix}
        0&0\\
        1&0
    \end{pmatrix}$,
where the first factor is a nontrivial idempotent and the second is a
nonzero nilpotent. 

Thus in every case
     $M(a,b)=E_1N_1$
with $E_1$ nontrivial and $N_1$ nonzero nilpotent.

It remains to prove the factorization in the opposite order. 
Observe that we have
    $M(a,b)^T
    =
    \begin{pmatrix}
        a&0\\
        b&0
    \end{pmatrix}
    \begin{pmatrix}
        1&0\\
        0&0
    \end{pmatrix}$,
so Lemma \ref{prodcontid} implies that
    $M(a,b)^T\in \OO_{M(c,d)}$
for some $c,d\in R$. By the first part of the proof, $M(c,d)$ is a
product of a nontrivial idempotent and a nonzero nilpotent. Since
conjugation preserves nonzero nilpotents and nontrivial idempotents,
there exist a nontrivial idempotent $F$ and a nonzero nilpotent $N$
such that
    $M(a,b)^T=FN$.
Taking transposes gives
    $M(a,b)=N^TF^T$.
Since transposition preserves nilpotency, nonzeroness and
idempotency, $N^T$ is a nonzero nilpotent and $F^T$ is a nontrivial
idempotent. Hence
    $M(a,b)=N_2E_2$
for suitable nonzero nilpotent $N_2$ and nontrivial idempotent $E_2$.
\end{proof}

%
%
%

\begin{Corollary}
Let $R$ be a finite commutative local principal ring. Then every
product of nilpotent and idempotent matrices which contains at least
one nontrivial idempotent factor is
both an IN-matrix and an NI-matrix. Moreover, it admits both a
nontrivial IN-factorization and a nontrivial NI-factorization.
\end{Corollary}
\begin{proof}
 Choose such a product $A$. Since a factorization of $A$ contains a
nontrivial idempotent, Lemma \ref{prodcontid} implies that
$A\in \OO_{M(a,b)}$
for some $a,b\in R$. Since conjugation preserves both nilpotency and
idempotency, it suffices to consider the case $A=M(a,b)$. The result thus follows from Lemma \ref{mabne}.
\end{proof}

We are now in a position to characterize the IN- and NI-matrices in
$M_2(R)$ and, in particular, to show that the two classes coincide. In fact, we obtain an
explicit description of their common class.

\begin{Theorem}
\label{main}
Let $R$ be a finite commutative local principal ring. Denote by
$\operatorname{IN}(M_2(R))$ and $\operatorname{NI}(M_2(R))$ the sets
of all IN- and NI-matrices in $M_2(R)$, respectively. Then
\[
    \operatorname{IN}(M_2(R))
    =
    \operatorname{NI}(M_2(R))
    =
    N(M_2(R))
    \cup
    \bigcup_{a,b\in R} \OO_{M(a,b)}.
\]
\end{Theorem}

\begin{proof}
We first prove that
    $\operatorname{IN}(M_2(R))
    \subseteq
    N(M_2(R))
    \cup
    \bigcup_{a,b\in R} \OO_{M(a,b)}$.
Let $A\in\operatorname{IN}(M_2(R))$. Then
$A=EN$
for some idempotent matrix $E\in M_2(R)$ and some nilpotent matrix
$N\in M_2(R)$.

If $E=0$, then $A=0$, and hence $A\in N(M_2(R))$. If $E=I$, then
$A=N$, so again $A\in N(M_2(R))$.
Suppose now that $E$ is a nontrivial idempotent. By Lemma \ref{prodcontid},
there exist $a,b\in R$
such that $A=EN\in \OO_{M(a,b)}$.
Consequently,
    $\operatorname{IN}(M_2(R))
    \subseteq
    N(M_2(R))
    \cup
    \bigcup_{a,b\in R}\OO_{M(a,b)}$.

Conversely, every nilpotent matrix $N$ is trivially an IN-matrix,
since $N=IN$.
Moreover, by Lemma \ref{mabne}, every matrix $M(a,b)$ is an IN-matrix.
Since the property of being an IN-matrix is preserved under
conjugation, every matrix in $\OO_{M(a,b)}$ is an IN-matrix. Therefore
    $N(M_2(R))
    \cup
    \bigcup_{a,b\in R}\OO_{M(a,b)}
    \subseteq
    \operatorname{IN}(M_2(R))$ and thus $\operatorname{IN}(M_2(R))
    =
    N(M_2(R))
    \cup
    \bigcup_{a,b\in R}\OO_{M(a,b)}$.

The argument for NI-matrices is analogous.
%
%
%
Consequently,
\[
    \operatorname{IN}(M_2(R))
    =
    \operatorname{NI}(M_2(R)) = N(M_2(R))  \cup  \bigcup_{a,b\in R}\OO_{M(a,b)},
\]
as claimed.
\end{proof}

\begin{Example}
The set of nilpotent matrices in the preceding theorem cannot, in
general, be omitted. Let $R=\mathbb Z_4$ and consider
\[
    A=
    \begin{pmatrix}
        2&2\\
        0&2
    \end{pmatrix}.
\]
Since $A^2=0$,
the matrix $A$ is nilpotent and is therefore trivially both an
IN-matrix and an NI-matrix.

We claim, however, that
    $A\notin
    \bigcup_{a,b\in\mathbb Z_4}\OO_{M(a,b)}$.
Suppose to the contrary that $A$ is conjugate to $M(a,b)$.
Reducing the conjugacy relation modulo $2$, and observing that
$\overline A=0$, we obtain
    $M(\overline a,\overline b)=0$.
Hence $a,b\in2\mathbb Z_4$. Moreover, invariance of the trace gives
    $a=\operatorname{tr}(M(a,b))
     =\operatorname{tr}(A)=0$.
Thus $A$ would have to be conjugate either to $M(0,0)$ or to
$M(0,2)$.
Clearly $A$ is not conjugate to $M(0,0)=0$. Furthermore,
    $|\ker A|=4$,
whereas
    $|\ker M(0,2)|=8$.
Since conjugate matrices have isomorphic kernels, this is impossible.
Therefore
    $A\notin
    \bigcup_{a,b\in\mathbb Z_4}\OO_{M(a,b)}$.

Thus $A$ is both an IN-matrix and an NI-matrix without belonging to
any of the orbits $\OO_{M(a,b)}$. This shows that the nilpotent term in
\[
    \operatorname{IN}(M_2(R))
    =
    \operatorname{NI}(M_2(R))
    =
    N(M_2(R))
    \cup
    \bigcup_{a,b\in R}\OO_{M(a,b)}
\]
is essential.
\end{Example}

\begin{Lemma}
    \label{orderofo}
    Let $R$ be a finite commutative local principal ring of cardinality
$q^n$, where $R/J(R)\cong GF(q)$. Then $|\bigcup_{a,b \in R} \OO_{M(a,b)}|=\frac{q^{3n}(q+1)^2-q}{q^2+q+1}$.
\end{Lemma}
\begin{proof}
    The proof of \cite[Theorem 3.8]{dolzanid} establishes that $|\bigcup_{a,b \in R} \OO_{M(a,b)}|+1=\frac{q^2+1+q^{3n}(q+1)^2}{q^2+q+1}$.
\end{proof}

\begin{Lemma}
\label{orderofnil}
Let $R$ be a finite commutative local ring of cardinality $q^n$, where
$R/J(R)\cong GF(q)$. Then
    $|N(M_2(R))|=q^{4n-2}$.
\end{Lemma}

\begin{proof}
Let
\[
    \pi:M_2(R)\longrightarrow M_2(R/J(R))\cong M_2(GF(q))
\]
be the natural projection. If $A\in M_2(R)$ is nilpotent then $\pi(A)$ is nilpotent. 
Conversely, if $\pi(A)$ is nilpotent, then $\pi(A)^2=0$, and hence $A^2\in M_2(J(R))$.
Since $J(R)$ is nilpotent, it follows that $A$ is nilpotent.

There are $q^2$ nilpotent matrices in $M_2(GF(q))$, and each of them
has $|J(R)|^4=q^{4n-4}$
preimages under $\pi$. Consequently, 
    $|N(M_2(R))|
    =q^2q^{4n-4}
    =q^{4n-2}$.
\end{proof}

\begin{Theorem}
Let $R$ be a finite commutative local principal ring of cardinality
$q^n$, where $R/J(R)\cong GF(q)$. Then
\[
    |\operatorname{IN}(M_2(R))|
    =
    |\operatorname{NI}(M_2(R))|
    =
    q^{3n-2}(q^n+q^2-1).
\]
\end{Theorem}

\begin{proof}
By Theorem \ref{main}, we have
    $\operatorname{IN}(M_2(R))
    =
    \operatorname{NI}(M_2(R))
    =
    N(M_2(R))
    \cup
    \bigcup_{a,b\in R}\OO_{M(a,b)}$.
Put
    $\mathcal O=\bigcup_{a,b\in R}\OO_{M(a,b)}$.
It follows that
    $|\operatorname{NI}(M_2(R))|
    =
    |N(M_2(R))|+|\mathcal O|
    -|N(M_2(R))\cap\mathcal O|$.
By Lemma \ref{orderofnil}, we have $|N(M_2(R))|=q^{4n-2}$.
Furthermore, Lemma \ref{orderofo} gives
    $|\mathcal O|
    =
    \frac{q^{3n}(q+1)^2-q}{q^2+q+1}$.

It remains to determine the intersection
$N(M_2(R))\cap\mathcal O$. Observe that, for every $m\geq1$,
    $M(a,b)^m
    =
    \begin{pmatrix}
        a^m&a^{m-1}b\\
        0&0
    \end{pmatrix}$.
Since the set of nilpotent elements of $R$ is $J(R)$, it follows that
\[
    M(a,b)\text{ is nilpotent}
    \quad\Longleftrightarrow\quad
    a\in J(R).
\]
As nilpotency is preserved under conjugation, we therefore have
\[
    N(M_2(R))\cap\mathcal O
    =
    \bigcup_{\substack{a\in J(R)\\b\in R}}\OO_{M(a,b)}.
\]

Let $J(R)=(x)$. For $a\neq0$, write
$a\in J(R)^l\setminus J(R)^{l+1}$ for some $0\le l\le n-1$.
If $b\neq0$, write
$b\in J(R)^k\setminus J(R)^{k+1}$ for some $0\le k\le n-1$,
and set $k=n$ when $b=0$.
   Using \cite[Corollary 3.7]{dolzanid}, we arrive at 
   $\OO_{M(a,b)}=\OO_{M(a,0)}$ if $k \geq l$ and $\OO_{M(a,b)}=\OO_{M(a,x^k)}$ if $k < l$. This implies
   $$\left|\bigcup_{b \in R} \OO_{M(a,b)}\right|=\left|\OO_{M(a,0)}\right|+\sum_{k=0}^{l-1} \left|\OO_{M(a,x^k)}\right|.$$
   From \cite[Lemma 3.5]{dolzanid} we therefore have 
   $$\left|\bigcup_{b \in R} \OO_{M(a,b)}\right|=q^{2n-2l-1}(q+1)+\sum_{k=0}^{l-1}q^{2(n-k-1)}(q^{2}-1), \text { if } a\neq 0, \text { and }$$
   $$\left|\bigcup_{b \in R} \OO_{M(0,b)}\right|=1+\sum_{k=0}^{n-1}q^{2(n-k-1)}(q^{2}-1).$$
      Thus summing the corresponding
disjoint orbits over $a\in J(R)$ (the orbits corresponding to distinct values of $a$ are disjoint
by invariance of the trace) and using Lemma \ref{pid}, we obtain 
   \begin{multline*}
         \left|
    \bigcup_{\substack{a\in J(R)\\b\in R}}\OO_{M(a,b)}
    \right|
    =1+\sum_{k=0}^{n-1}q^{2(n-k-1)}(q^{2}-1)+\\ \sum_{l=1}^{n-1}\left(q^{n-1-l}(q-1)\left(q^{2n-2l-1}(q+1)+\sum_{k=0}^{l-1}q^{2(n-k-1)}(q^{2}-1)\right)\right).
   \end{multline*}
A direct calculation, using
\[
    \sum_{k=0}^{l-1}q^{2(n-k-1)}(q^2-1)
    =q^{2n}-q^{2n-2l},
\]
shows that
\[
\begin{aligned}
&1+\sum_{k=0}^{n-1}q^{2(n-k-1)}(q^2-1)
+\sum_{l=1}^{n-1}q^{n-1-l}(q-1)
\left(q^{2n-2l-1}(q+1)
+\sum_{k=0}^{l-1}q^{2(n-k-1)}(q^2-1)\right)\\
&\qquad
=q^{2n}
+(q-1)\sum_{l=1}^{n-1}q^{3n-3l-2}(q^{2l+1}+1)\\
&\qquad
=q^{3n-1}
+\frac{q(q^{3n-3}-1)}{q^2+q+1}\\
&\qquad
=\frac{q^{3n-2}(q+1)(q^2+1)-q}{q^2+q+1}.
\end{aligned}
\]
Consequently,
\[
\begin{aligned}
    |\operatorname{NI}(M_2(R))|
    &=
    q^{4n-2}
    +
    \frac{q^{3n}(q+1)^2-q}{q^2+q+1}\\
    &\qquad
    -
    \frac{q^{3n-2}(q+1)(q^2+1)-q}
         {q^2+q+1}\\
    &=
    q^{3n-2}(q^n+q^2-1).
\end{aligned}
\]
Since $\operatorname{IN}(M_2(R))
    =
    \operatorname{NI}(M_2(R))$,
the same formula holds for the number of IN-matrices.
\end{proof}

\bigskip

{\bf Statements and Declarations.} The author states that there are no competing interests. 

\bigskip

\bibliographystyle{amsplain}
\bibliography{biblio}

\bigskip

\end{document}